\documentclass{amsart}
\usepackage{amsmath, amssymb,mathtools,url,color}
\title{Fattening Up Warning's Second Theorem}
\author{Pete L. Clark}
\begin{document}
\newtheorem{lemma}{Lemma}[section]
\newtheorem{prop}[lemma]{Proposition}
\newtheorem{cor}[lemma]{Corollary}
\newtheorem{thm}[lemma]{Theorem}
\newtheorem{ques}{Question}
\newtheorem{quest}[lemma]{Question}
\newtheorem{conj}[lemma]{Conjecture}
\newtheorem{fact}[lemma]{Fact}
\newtheorem*{mainthm}{Main Theorem}
\newtheorem{obs}[lemma]{Observation}
\newtheorem{hint}{Hint}
\newtheorem{remark}[lemma]{Remark}
\newtheorem{example}[lemma]{Example}
\newtheorem{exercise}{Exercise}

\maketitle

\begin{abstract}
We present a generalization of Warning's Second Theorem to polynomial systems over a finite local principal ring with suitably restricted input and output variables.  This generalizes a recent result with Forrow and Schmitt (and gives a new proof
of that result).  Applications to additive
group theory, graph theory and polynomial interpolation are pursued in detail.
\end{abstract}

\tableofcontents

\newcommand{\pp}{\mathfrak{p}}
\newcommand{\DD}{\mathcal{D}}
\newcommand{\F}{\ensuremath{\mathbb F}}
\newcommand{\Fp}{\ensuremath{\F_p}}
\newcommand{\Fl}{\ensuremath{\F_l}}
\newcommand{\Fpbar}{\overline{\Fp}}
\newcommand{\Fq}{\ensuremath{\F_q}}
\newcommand{\PP}{\mathbb{P}}
\newcommand{\PPone}{\mathfrak{p}_1}
\newcommand{\PPtwo}{\mathfrak{p}_2}
\newcommand{\PPonebar}{\overline{\PPone}}
\newcommand{\N}{\ensuremath{\mathbb N}}
\newcommand{\Q}{\ensuremath{\mathbb Q}}
\newcommand{\Qbar}{\overline{\Q}}
\newcommand{\R}{\ensuremath{\mathbb R}}
\newcommand{\Z}{\ensuremath{\mathbb Z}}
\newcommand{\SSS}{\ensuremath{\mathcal{S}}}
\newcommand{\Rn}{\ensuremath{\mathbb R^n}}
\newcommand{\Ri}{\ensuremath{\R^\infty}}
\newcommand{\C}{\ensuremath{\mathbb C}}
\newcommand{\Cn}{\ensuremath{\mathbb C^n}}
\newcommand{\Ci}{\ensuremath{\C^\infty}}\newcommand{\U}{\ensuremath{\mathcal U}}
\newcommand{\gn}{\ensuremath{\gamma^n}}
\newcommand{\ra}{\ensuremath{\rightarrow}}
\newcommand{\fhat}{\ensuremath{\hat{f}}}
\newcommand{\ghat}{\ensuremath{\hat{g}}}
\newcommand{\hhat}{\ensuremath{\hat{h}}}
\newcommand{\covui}{\ensuremath{\{U_i\}}}
\newcommand{\covvi}{\ensuremath{\{V_i\}}}
\newcommand{\covwi}{\ensuremath{\{W_i\}}}
\newcommand{\Gt}{\ensuremath{\tilde{G}}}
\newcommand{\gt}{\ensuremath{\tilde{\gamma}}}
\newcommand{\Gtn}{\ensuremath{\tilde{G_n}}}
\newcommand{\gtn}{\ensuremath{\tilde{\gamma_n}}}
\newcommand{\gnt}{\ensuremath{\gtn}}
\newcommand{\Gnt}{\ensuremath{\Gtn}}
\newcommand{\Cpi}{\ensuremath{\C P^\infty}}
\newcommand{\Cpn}{\ensuremath{\C P^n}}
\newcommand{\lla}{\ensuremath{\longleftarrow}}
\newcommand{\lra}{\ensuremath{\longrightarrow}}
\newcommand{\Rno}{\ensuremath{\Rn_0}}
\newcommand{\dlra}{\ensuremath{\stackrel{\delta}{\lra}}}
\newcommand{\pii}{\ensuremath{\pi^{-1}}}
\newcommand{\la}{\ensuremath{\leftarrow}}
\newcommand{\gonem}{\ensuremath{\gamma_1^m}}
\newcommand{\gtwon}{\ensuremath{\gamma_2^n}}
\newcommand{\omegabar}{\ensuremath{\overline{\omega}}}
\newcommand{\dlim}{\underset{\lra}{\lim}}
\newcommand{\ilim}{\operatorname{\underset{\lla}{\lim}}}
\newcommand{\Hom}{\operatorname{Hom}}
\newcommand{\Ext}{\operatorname{Ext}}
\newcommand{\Part}{\operatorname{Part}}
\newcommand{\Ker}{\operatorname{Ker}}
\newcommand{\im}{\operatorname{im}}
\newcommand{\ord}{\operatorname{ord}}
\newcommand{\unr}{\operatorname{unr}}
\newcommand{\B}{\ensuremath{\mathcal B}}
\newcommand{\Ocr}{\ensuremath{\Omega_*}}
\newcommand{\Rcr}{\ensuremath{\Ocr \otimes \Q}}
\newcommand{\Cptwok}{\ensuremath{\C P^{2k}}}
\newcommand{\CC}{\ensuremath{\mathcal C}}
\newcommand{\gtkp}{\ensuremath{\tilde{\gamma^k_p}}}
\newcommand{\gtkn}{\ensuremath{\tilde{\gamma^k_m}}}
\newcommand{\QQ}{\ensuremath{\mathcal Q}}
\newcommand{\I}{\ensuremath{\mathcal I}}
\newcommand{\sbar}{\ensuremath{\overline{s}}}
\newcommand{\Kn}{\ensuremath{\overline{K_n}^\times}}
\newcommand{\tame}{\operatorname{tame}}
\newcommand{\Qpt}{\ensuremath{\Q_p^{\tame}}}
\newcommand{\Qpu}{\ensuremath{\Q_p^{\unr}}}
\newcommand{\scrT}{\ensuremath{\mathfrak{T}}}
\newcommand{\That}{\ensuremath{\hat{\mathfrak{T}}}}
\newcommand{\Gal}{\operatorname{Gal}}
\newcommand{\Aut}{\operatorname{Aut}}
\newcommand{\tors}{\operatorname{tors}}
\newcommand{\Zhat}{\hat{\Z}}
\newcommand{\linf}{\ensuremath{l_\infty}}
\newcommand{\Lie}{\operatorname{Lie}}
\newcommand{\GL}{\operatorname{GL}}
\newcommand{\End}{\operatorname{End}}
\newcommand{\aone}{\ensuremath{(a_1,\ldots,a_k)}}
\newcommand{\raone}{\ensuremath{r(a_1,\ldots,a_k,N)}}
\newcommand{\rtwoplus}{\ensuremath{\R^{2  +}}}
\newcommand{\rkplus}{\ensuremath{\R^{k +}}}
\newcommand{\length}{\operatorname{length}}
\newcommand{\Vol}{\operatorname{Vol}}
\newcommand{\cross}{\operatorname{cross}}
\newcommand{\GoN}{\Gamma_0(N)}
\newcommand{\GeN}{\Gamma_1(N)}
\newcommand{\GAG}{\Gamma \alpha \Gamma}
\newcommand{\GBG}{\Gamma \beta \Gamma}
\newcommand{\HGD}{H(\Gamma,\Delta)}
\newcommand{\Ga}{\mathbb{G}_a}
\newcommand{\Div}{\operatorname{Div}}
\newcommand{\Divo}{\Div_0}
\newcommand{\Hstar}{\cal{H}^*}
\newcommand{\txon}{\tilde{X}_0(N)}
\newcommand{\sep}{\operatorname{sep}}
\newcommand{\notp}{\not{p}}
\newcommand{\Aonek}{\mathbb{A}^1/k}
\newcommand{\Wa}{W_a/\mathbb{F}_p}
\newcommand{\Spec}{\operatorname{Spec}}

\newcommand{\abcd}{\left[ \begin{array}{cc}
a & b \\
c & d
\end{array} \right]}

\newcommand{\abod}{\left[ \begin{array}{cc}
a & b \\
0 & d
\end{array} \right]}

\newcommand{\unipmatrix}{\left[ \begin{array}{cc}
1 & b \\
0 & 1
\end{array} \right]}

\newcommand{\matrixeoop}{\left[ \begin{array}{cc}
1 & 0 \\
0 & p
\end{array} \right]}

\newcommand{\w}{\omega}
\newcommand{\Qpi}{\ensuremath{\Q(\pi)}}
\newcommand{\Qpin}{\Q(\pi^n)}
\newcommand{\pibar}{\overline{\pi}}
\newcommand{\pbar}{\overline{p}}
\newcommand{\lcm}{\operatorname{lcm}}
\newcommand{\trace}{\operatorname{trace}}
\newcommand{\OKv}{\mathcal{O}_{K_v}}
\newcommand{\Abarv}{\tilde{A}_v}
\newcommand{\kbar}{\overline{k}}
\newcommand{\Kbar}{\overline{K}}
\newcommand{\pl}{\rho_l}
\newcommand{\plt}{\tilde{\pl}}
\newcommand{\plo}{\pl^0}
\newcommand{\Du}{\underline{D}}
\newcommand{\A}{\mathbb{A}}
\newcommand{\D}{\underline{D}}
\newcommand{\op}{\operatorname{op}}
\newcommand{\Glt}{\tilde{G_l}}
\newcommand{\gl}{\mathfrak{g}_l}
\newcommand{\gltwo}{\mathfrak{gl}_2}
\newcommand{\sltwo}{\mathfrak{sl}_2}
\newcommand{\h}{\mathfrak{h}}
\newcommand{\tA}{\tilde{A}}
\newcommand{\sss}{\operatorname{ss}}
\newcommand{\X}{\Chi}
\newcommand{\ecyc}{\epsilon_{\operatorname{cyc}}}
\newcommand{\hatAl}{\hat{A}[l]}
\newcommand{\sA}{\mathcal{A}}
\newcommand{\sAt}{\overline{\sA}}
\newcommand{\OO}{\mathcal{O}}
\newcommand{\OOB}{\OO_B}
\newcommand{\Flbar}{\overline{\F_l}}
\newcommand{\Vbt}{\widetilde{V_B}}
\newcommand{\XX}{\mathcal{X}}
\newcommand{\GbN}{\Gamma_\bullet(N)}
\newcommand{\Gm}{\mathbb{G}_m}
\newcommand{\Pic}{\operatorname{Pic}}
\newcommand{\FPic}{\textbf{Pic}}
\newcommand{\solv}{\operatorname{solv}}
\newcommand{\Hplus}{\mathcal{H}^+}
\newcommand{\Hminus}{\mathcal{H}^-}
\newcommand{\HH}{\mathcal{H}}
\newcommand{\Alb}{\operatorname{Alb}}
\newcommand{\FAlb}{\mathbf{Alb}}
\newcommand{\gk}{\mathfrak{g}_k}
\newcommand{\car}{\operatorname{char}}
\newcommand{\Br}{\operatorname{Br}}
\newcommand{\gK}{\mathfrak{g}_K}
\newcommand{\coker}{\operatorname{coker}}
\newcommand{\red}{\operatorname{red}}
\newcommand{\CAY}{\operatorname{Cay}}
\newcommand{\ns}{\operatorname{ns}}
\newcommand{\xx}{\mathbf{x}}
\newcommand{\yy}{\mathbf{y}}
\newcommand{\E}{\mathbb{E}}
\newcommand{\rad}{\operatorname{rad}}
\newcommand{\Top}{\operatorname{Top}}
\newcommand{\Map}{\operatorname{Map}}
\newcommand{\Li}{\operatorname{Li}}
\renewcommand{\Map}{\operatorname{Map}}
\newcommand{\ZZ}{\mathcal{Z}}
\newcommand{\uu}{\mathfrak{u}}
\newcommand{\mm}{\mathfrak{m}}
\newcommand{\zz}{\mathbf{z}}
\newcommand{\Aff}{\operatorname{Aff}}
\newcommand{\MaxSpec}{\operatorname{MaxSpec}}
\newcommand{\qq}{\mathfrak{q}}
\newcommand{\supp}{\operatorname{supp}}
\newcommand{\Lceil}{\bigg{\lceil}}
\newcommand{\Rceil}{\bigg{\rceil}}
\newcommand{\rr}{\mathfrak{r}}
\renewcommand{\aa}{\mathfrak{a}}
\newcommand{\bb}{\mathfrak{b}}
\newcommand{\GG}{\mathcal{G}}
\renewcommand{\AA}{{\bf{A}}}
\newcommand{\BB}{{\bf{B}}}
\newcommand{\YY}{{\bf{Y}}}
\renewcommand{\tt}{{\bf{t}}}
\renewcommand{\qq}{{\bf{q}}}

\noindent

\section{Introduction}

\subsection{Notation and Terminology}
\textbf{} \\ \\ \noindent
Let $n,a_1,\ldots,a_n \in \Z^+$ and let $1 \leq N \leq \sum_{i=1}^n a_i$.  As in \cite[$\S 2.1$]{Clark-Forrow-Schmitt14}, we put
\[ \mathfrak{m}(a_1,\ldots,a_n;N) = \begin{cases} 1 & N < n \\
\min
 \prod_{i=1}^n y_i & n \leq N \leq \sum_{i=1}^n a_i\end{cases};\]
the minimum is over $(y_1,\ldots,y_n) \in \Z^n$ with $y_i \in [1,a_i]$
for all $i$ and $\sum_{i=1}^n y_i = N$.
\\  \\
Let $R$ be a ring, $B \subset R$ a subset, $I$ an ideal of $R$, and $x \in R$.  We write ``$x \in B \pmod{I}$''
to mean that there is $b \in B$ such that $x-b \in I$.
\\ \\
Let $R$ be a ring.  As in \cite{Clark14}, we say a subset $A \subset R$ satisfies \textbf{Condition (F)} (resp. \textbf{Condition (D)}) if $A$ is nonempty, finite and for
any distinct elements $x,y \in A$, $x-y$ is a unit in $R$ (resp. is not a zero-divisor in $R$).

\subsection{Prior Results}
\textbf{} \\ \\ \noindent
We begin with the results of Chevalley and Warning.

\begin{thm}
\label{WARTHM}
Let $n,r,d_1,\ldots,d_r \in \Z^+$ with
$d := d_1 + \ldots + d_r < n$.
For $1 \leq i \leq r$, let
$f_i(t_1,\ldots,t_n) \in \F_q[\tt] = \F_q[t_1,\ldots,t_n]$ be a polynomial of degree $d_i$.  Let
\[ Z = Z(f_1,\ldots,f_r) = \{x \in \F_q^n \mid f_1(x) = \ldots = f_r(x) = 0\}. \]
a) (Chevalley's Theorem \cite{Chevalley35}) We have $\# Z = 0$ or $\# Z \geq 2$.  \\
b) (Warning's Theorem \cite{Warning35}) We have $\# Z \equiv 0 \pmod{p}$. \\
c) (Warning's Second Theorem \cite{Warning35}) We have $\# Z = 0 \text{ or } \# Z \geq q^{n-d}$.
\end{thm}
\noindent
Chevalley's proof of Theorem \ref{WARTHM}a) can be easily
modified to yield Theorem \ref{WARTHM}b).  Warning's real contribution was Theorem \ref{WARTHM}c), a result which has, I feel, been too little appreciated.
%
It is sharp in the following strong sense: for
any $d_1,\ldots,d_r \in \N$ with $d := d_1+\ldots+d_r < n$, there are $f_1,\ldots,f_r \in \F_q[\tt]$ with $\deg f_i = d_i$ for all $1 \leq i \leq r$ such that $\# Z(f_1,\ldots,f_r) = q^{n-d}$.  One can build
such examples by combining norm forms associated to field extensions $\F_{q^a}/\F_q$ and linear polynomials.  On the other hand, although in these
examples the equations are generally nonlinear, the solution sets are still
affine subspaces.  In \cite{Heath-Brown11}, Heath-Brown showed that under the hypotheses
of Theorem \ref{WARTHM}c), when $Z$ is nonempty and is not an affine subspace
of $\F_q^n$ one always has $\# Z > q^{n-d}$, and in fact $\# Z \geq 2 q^{n-d}$
for all $q \geq 4$.
\\ \\
Apart from \cite{Heath-Brown11} there had been little
further exploration of Theorem \ref{WARTHM}c) until \cite{Clark-Forrow-Schmitt14}, in which A. Forrow, J.R. Schmitt and I established the
following result.

\begin{thm}(Restricted Variable Warning's Second Theorem \cite{Clark-Forrow-Schmitt14})
\label{RVW2THM}
\label{RVW2T}
Let $K$ be a number field with ring of integers $R$, let $\pp$ be a nonzero prime ideal of $R$, and let $q = p^{\ell}$ be the prime power such that
$R/\pp \cong \F_q$.  Let $A_1,\ldots,A_n$ be nonempty subsets of $R$ such that for each $i$,
the elements of $A_i$ are pairwise incongruent modulo $\pp$, and put $\AA = \prod_{i=1}^n A_i$.  Let $r,v_1,\ldots,v_r \in \Z^+$.  Let $P_1,\ldots,P_r \in R[t_1,\ldots,t_n]$.  Let
\[ Z_{\AA} = \{ x \in \AA \mid P_j(x) \equiv 0 \pmod{\pp^{v_j}} \ \forall 1 \leq j \leq r\}, \ \zz_{\AA} = \# Z_{\AA}. \]
Then $\zz_{\AA} = 0$ or $\zz_{\AA} \geq \mm \left( \# A_1,\ldots, \# A_n; \# A_1 + \ldots + \# A_n -
\sum_{j=1}^r (q^{v_j} -1) \deg(P_j) \right)$.
\end{thm}
\noindent
This generalizes Theorem \ref{WARTHM}c) in two directions: first,
instead of working over finite fields, we work modulo
powers of a prime ideal in the ring of integers of a number field.  In the case $K = \Q$ we are studying systems of congruence modulo (varying) powers of a (fixed) prime $p$.  Second, we study solutions
in which each variable is independently restricted to a finite subset of $\Z_K$ satisfying the condition that no two distinct elements are congruent modulo $\pp$.
\\ \indent
These extensions appear already in work of Schanuel \cite{Schanuel74}, Baker-Schmidt \cite{Baker-Schmidt80}, Schauz \cite{Schauz08a}, Wilson \cite{Wilson06} and Brink \cite{Brink11}.  They are largely motivated by applications to combinatorics.  For combinatorial applications we work over $K = \Q$ and get congruences modulo powers of $p$.  The most classical applications concern the case in which each variable is restricted to take values $0$ and $1$.  More recently there has been a surge of interest in more general subsets $A_i$: this yields weighted analogues
of the more classical combinatorial problems.
\\ \indent
The previous works
used either \emph{ad hoc} methods or Alon's Combinatorial Nullstellensatz and yielded \emph{nonuniqueness theorems}: results with conclusion ``there cannot be exactly one solution''.  To prove Theorem \ref{RVW2T} we instead applied the Alon-F\"uredi Theorem, which yields a lower bound on the number of solutions in terms of the quantity $\mm(a_1,\ldots,a_n;N)$.  To collapse this type of result to a nonuniqueness
theorem one simply uses the Pigeonhole Principle
\begin{equation}
\label{PHP}
\mm(a_1,\ldots,a_n;N) \geq 2 \iff
N > n.
\end{equation}
Applying (\ref{PHP}) to Theorem \ref{RVW2T}, one recovers a result of Brink.

\begin{cor}(Brink \cite{Brink11})
\label{NFBRINKTHM}
Let $K$ be a number field with ring of integers $R$, let $\pp$ be a nonzero prime ideal of $R$, and let $q = p^{\ell}$ be the prime power such that
$R/\pp \cong \F_q$.  Let $P_1(t_1,\ldots,t_n),\ldots,P_r(t_1,\ldots,t_n) \in R[t_1,\ldots,t_n]$, let $v_1,\ldots,v_r \in \Z^+$, and
let $A_1,\ldots,A_n$ be nonempty subsets of $R$ such that for each $i$,
the elements of $A_i$ are pairwise incongruent modulo $\pp$, and put
$\AA = \prod_{i=1}^n A_i$.  Let
\[ Z_{\AA} = \{x \in {\AA} \mid P_j(x) \equiv 0 \pmod{\pp^{v_j}} \ \forall 1 \leq j \leq r \}, \ \zz_{\AA} = \# Z_{\AA}. \]
 If $\sum_{j=1}^r (q^{v_j}-1)\deg(P_j) < \sum_{i=1}^n \left( \#A_i - 1 \right)$, then $\zz_{\AA} \neq 1$. 
\end{cor}
\noindent
The case of $K = \Q$ was independently (in fact, earlier) established  by U. Schauz and R. Wilson, so we call this result the \textbf{Schauz-Wilson-Brink Theorem}.  If we further specialize to $A_i = \{0,1\}$ for all $i$ we recover
\textbf{Schanuel's Theorem}.

\subsection{The Main Theorem}
\textbf{} \\ \\ \noindent
For the convenience of readers who are primarily interested in combinatorial applications, we state the main result of this paper first in a special case.


\begin{thm}
\label{MAINTHMZ}
Let $p$ be a prime, let $n,r,v \in \Z^+$, and for $1 \leq i \leq r$,
let $1 \leq v_ j \leq v$.  Let
$A_1,\ldots,A_n,B_1,\ldots,B_r \subset \Z/p^v \Z$ be nonempty subsets each having the property that no two distinct elements are congruent modulo $p$.
Let $f_1,\ldots,f_r \in \Z/p^v\Z[\tt] = \Z/p^v\Z[t_1,\ldots,t_n]$.
Let
\[ Z_{\AA}^{\BB} = \{x \in \prod_{i=1}^n A_i \mid \forall 1 \leq j \leq r, \ f_j(x) \in B_j \pmod{p^{v_j}} \}. \]
Then $Z_{\AA}^{\BB} = \varnothing$ or
\[ \# Z_{\AA}^{\BB} \geq \mm\left(\# A_1,\ldots,\# A_n; \sum_{i=1}^n \# A_i - \sum_{j=1}^r
(p^{v_j}-\# B_j) \deg f_j \right) . \]
\end{thm}

\begin{cor}
\label{MAINCORZ}
Maintain the setup of Theorem \ref{MAINTHMZ}. \\
a) If $A_i = \{0,1\}$ for all $1 \leq i \leq n$, then $Z_{\AA}^{\BB} = \varnothing$ or
\[ \# Z_{\AA}^{\BB} \geq 2^{n-\sum_{j=1}^r (p^{v_j} -\# B_j) \deg(f_j)}. \]
b) If $A_i = \{0,1\}$ for all $i$ and $f_j(0) = 0 \in B_j$ for all $j$, there is $0 \neq x \in Z_{\AA}^{\BB}$
if
\[n > \sum_{j=1}^r (p^{v_j} -\# B_j) \deg(f_j). \]
\end{cor}
\begin{proof}
Applying Theorem \ref{MAINTHMZ} in the case $A_1 = \ldots = A_n = \{0,1\}$ and using the fact that for any $0 \leq k \leq n$, we have $\mm(2,\ldots,2;2n-k) = 2^{n-k}$  \cite[Lemma 2.2c)]{Clark-Forrow-Schmitt14}, we get
part a).  Combining with (\ref{PHP}) we get part b).
\end{proof}
\noindent
If in Corollary \ref{MAINCORZ}b) we further require that all the polynomials
are linear, we recover a result of Alon-Friedland-Kalai \cite[Thm. A.1]{Alon-Friedland-Kalai84}.  For some (not all) combinatorial
applications linear polynomials are sufficient, and
$A_i = \{0,1\}$ corresponds to the ``unweighted'' combinatorial setup.
In this setting we see that the advantage of Corollary \ref{MAINCORZ}a) over part b) is directly analogous to that of Theorem \ref{RVW2T} over Brink's Theorem, namely a quantitative refinement of Alon-F\"uredi type.  In fact this gives an accurate glimpse of our method of proof of the
Main Theorem: we will establish and apply suitably generalized versions of a valuation-theoretic lemma of Alon-Friedland-Kalai and of the Alon-F\"uredi
Theorem.
\\ \\
To state the full version of the Main Theorem we need some algebraic preliminaries.  A \textbf{principal ring} is a commutative ring in which
every ideal is principal.  A ring is \textbf{local} if it has exactly one
maximal ideal.  Let $(\rr,\pp)$ be a local principal ring with maximal
ideal $\pp = (\pi)$.  By Nakayama's Lemma, $\bigcap_{i \geq 0} \pp^i = (0)$, so
for every nonzero $x \in \rr$, there is a unique $i \in \N$ such that
$x \in \pp^i \setminus \pp^{i+1}$, so $x = \pi^i y$ and $y$ is a unit
in $\rr$, so $(x) = (\pi^i) = \pp^i$.  Thus every nonzero ideal of $\rr$
is of the form $\pp^i$ for some $i \in \N$.  There are two possibilities:
\\ \\
(i) For all $a \in \Z^+$, $\pp^a \neq 0$.  Then $\rr$ is a DVR. \\
(ii) There is a positive integer $v$, the \textbf{length} of $\rr$, such that $\pp^{v-1} \neq (0)$ and $\pp^{v} = (0)$.
\\ \\
If $\rr$ is moreover finite then (ii) must hold.  Thus in any (nonzero) finite principal ring $(\rr,\pp)$ there is a positive integer $v$ such that the ideals of $\rr$ are \[\rr = \pp^0 \supsetneq \pp \supsetneq \pp^1 \supsetneq \ldots \supsetneq \pp^v = (0). \]

\begin{thm}
\label{MAINTHMLOCAL}
Let $(\rr,\pp)$ be a finite local ring of length $v$ and with residue field
$\rr/\pp \cong \F_q$.  Let $n,r \in \Z^+$, and for $1 \leq j \leq r$, let
$1 \leq v_j \leq v$.  Let $\aa_1,\ldots,\aa_n,\bb_1,\ldots,\bb_r \subset \rr$ be nonempty
subsets each having the property that no two distinct elements are congruent modulo $\pp$.  Let $f_1,\ldots,f_r \in \rr[\tt] = \rr[t_1,\ldots,t_n]$.
Let
\[ \mathfrak{z}_{\aa,\bb} = \{x \in \prod_{i=1}^n \aa_i \mid \forall 1 \leq j \leq r, \ f_j(x) \in \bb_j \pmod{\pp^{v_j}} \}. \]
Then $\mathfrak{z}_{\aa,\bb} = \varnothing$ or
\[ \# \mathfrak{z}_{\aa,\bb} \geq \mm\left(\# \aa_1,\ldots,\# \aa_n; \sum_{i=1}^n \# \aa_i - \sum_{j=1}^r
(q^{v_j}-\# \bb_j) \deg f_j \right) . \]
\end{thm}
\noindent
Consider the following variant of Theorem \ref{MAINTHMLOCAL}.

\begin{thm}
\label{MAINTHM}
Let $R$ be a Dedekind domain with maximal ideal $\pp$ and finite residue field
$R/\pp \cong \F_q$. Let $n,r,v_1,\ldots,v_r \in \Z^+$.  Let $A_1,\ldots,A_n,B_1,\ldots,B_r \subset R$ be nonempty subsets
each having the property that no two distinct elements are congruent modulo $\pp$. Let $r,v_1,\ldots,v_r \in \Z^+$.  Let $f_1,\ldots,f_r \in R[t_1,\ldots,t_n]$.  Put
\[ Z_{\AA}^{\BB} = \left\{ x \in \prod_{i=1}^n A_i \mid \forall 1 \leq j \leq r, f_j(x) \in B_j \pmod{\pp^{v_j}} \right\}.\]
Then $\# Z_{\AA}^{\BB} = 0$ or \[\# Z_{\AA}^{\BB} \geq \mm \left(\# A_1,\ldots,\# A_n; \sum_{i=1}^n  \# A_i
-\sum_{j=1}^r (q^{v_j} -\# B_j) \deg(f_j) \right).  \]
\end{thm}

\begin{remark}
\label{LOCALIZATIONREMARK}
By replacing $R$ with $R_{\pp}$, one immediately reduces the statement of
Theorem \ref{MAINTHM} to the case in which $R$ is a discrete valuation ring.
In this setting, the hypothesis on the $A_i$ and $B_j$ is simply Condition (F).
\end{remark}

\begin{prop}
Theorems \ref{MAINTHMLOCAL} and \ref{MAINTHM} are equivalent.
\end{prop}
\begin{proof}
Theorem \ref{MAINTHMLOCAL} $\implies$ Theorem \ref{MAINTHM}: let $\rr = R/\pp^v$ and let $q: R \ra \rr$ be the quotient map.  For $1 \leq i \leq n$,  let $\aa_i = q(A_i)$; for $1 \leq j \le r$, let $\overline{f_j} = q(f_j)$ and $\bb_j = q(B_j)$.  Then $\deg \overline{f_j} \leq \deg f_j$.  The hypothesis that
no two distinct elements of any one of these sets are congruent modulo $\pp$
ensures $\# \aa_i = \# A_i$ and $\# \bb_j = \# B_j$.  Applying Theorem \ref{MAINTHMLOCAL}
to $\rr$,$\aa_1,\ldots,\aa_n,\bb_1,\ldots,\bb_r$,$v_1,\ldots,v_r$,$\overline{f_1},\ldots,\overline{f_r}$
gives
\[ \# Z_{\AA}^{\BB} = \# \mathfrak{z}_{\aa,\bb} \geq \mm \left(\# \aa_1,\ldots,\# \aa_n; \sum_{i=1}^n  \# \aa_i
-\sum_{j=1}^r (q^{v_j} -\# \bb_j) \deg(\overline{f_j}) \right)\]
\[ \geq \mm \left(\# A_1,\ldots,\# A_n; \sum_{i=1}^n  \# A_i
-\sum_{j=1}^r (q^{v_j} -\# B_j) \deg(f_j) \right). \]
Theorem \ref{MAINTHM} $\implies$ Theorem \ref{MAINTHMLOCAL}: the Cohen structure theorems imply that an Artinian local principal ring is a quotient of a Dedekind domain (equivalently, of a DVR)  \cite[Cor. 11]{Hungerford68}.  Thus we may write $\rr = R/\pp^v$ for a Dedekind domain $R$.  We may lift $A_1,\ldots,A_n$,$B_1,\ldots,B_r$,$f_1,\ldots,f_r$ from $\rr$ to $R$ so as to preserve the sizes of the sets and the degrees of the polynomials.  Apply Theorem \ref{MAINTHM}.
\end{proof}

\begin{example}
Let $\rr$ be a finite commutative ring, and let $A \subset \rr$ satisfy Condition (D).
For $1 \leq i \leq r$, let $f_j \in \rr[t_1]$ be a univariate polynomial of degree $d_i \geq 0$, and let $B_1,\ldots,B_r \subset \rr$
be finite and nonempty.  Let
\[ \mathfrak{z}_A^{\BB} = \{ x \in A \mid f_1(x) \in B_1,\ldots,f_r(x) \in B_r \}. \]
Suppose first that $d_j \geq 1$ for all $j$.  Then for each $y \in \rr$, the polynomial
$f_j-y$ also has degree $d_j$, and because $A$ satisfies Condition (D), there are at most $\deg(f_j)$ elements of $A$ such that $f_j(x) = y$.  So there are at most $(\# \rr - \# B_j)(\deg f_j)$ elements $x \in A$ such that $f_j(x) \notin B_j$ and
thus
\[ \# \mathfrak{z}_A^{\BB} \geq \# A - \sum_{j=1}^r (\# \rr - \# B_j)(\deg f_j). \]
Now suppose that some $f_j$ is constant.  Then: if the constant value lies in $B_j$
then $f_j(A) \subset B_j$, whereas if the constant value does not lie in $B_j$ then
$\mathfrak{z}_A^{\BB} = \varnothing$.  \\ \indent
This establishes a stronger result than Theorem \ref{MAINTHMLOCAL} when $n = 1$.  In particular: the finite ring $\rr$ need not be local and principal, the target sets
$B_1,\ldots,B_r$ need not satisfy Condition (F) but rather may be arbitrary nonempty
subsets, and we do not need to separately allow $\mathfrak{z}_{\AA}^B = \varnothing$ if each polynomial
has positive degree.
\end{example}

\subsection{Comparison With Theorem \ref{RVW2T}}
\textbf{} \\ \\ \noindent
Theorem \ref{RVW2T} is the special case of Theorem \ref{MAINTHM}
obtained by taking $R = \Z_K$ and $B_j = \{0\}$ for all $j$.  Thus on the face of
it Theorem \ref{MAINTHM} is a twofold generalization of Theorem \ref{RVW2T}: in place
of $\Z_K$ we may take any pair $(R,\pp)$ with $R$ a Dedekind domain and $\pp$ a prime
ideal such that $R/\pp$ is a finite field; and in place of polynomial congruences
we are studying polynomial systems with \textbf{restricted output sets} $B_j$.
\\ \\
The first generalization turns out not to be an essential one.  Theorem \ref{MAINTHMLOCAL} shows that the result can be phrased in terms of finite, local principal rings.  But every finite local principal ring is isomorphic to $\Z_K/\pp^v$ for some prime ideal $\pp$ in the ring of
integers of a number field $K$.  This is due to A.A. Ne\v caev \cite{Necaev71}.  A more streamlined proof appears in \cite{Brunyate-Clark15}.

\begin{example}
Consider $\rr = \F_p[t]/(t^2)$: it is a finite, local principal ring
with residue cardinality $p$ and length $2$.  Further, it is a commutative $\F_p$-algebra of dimension $2$ which is not reduced: i.e., it has nonzero nilpotent elements, and this latter description characterizes $\rr$ up to isomorphism.  So
let $K = \Q(\sqrt{p})$ and let $\pp$ be the unique prime ideal of $\Z_K$ dividing $p$.
The ring $\Z_K/p \Z_K = \Z_K/\pp^2$ is also a commutative $\F_p$-algebra of dimension $2$ which is not reduced, so $\rr = \F_p[t]/(t^2) \cong \Z_K/\pp^2$.
\end{example}
\noindent
Nevertheless it is natural to think in terms of Dedekind
domains, and switching from one Dedekind domain to another seems artificial.  The proof of Theorem \ref{RVW2T} uses the fact that $\Z_K$ has characteristic zero in an essential way: a key technical tool was the use of \textbf{Schanuel-Brink operators} to replace a congruence modulo $\pp^v$ in $\Z_K$
with a system of congruences modulo $\pp$.  As Schanuel pointed out, this construction is morally about \textbf{Witt vectors} and thus particular to  unequal characteristic.  Our proof of Theorem \ref{MAINTHM} does not reduce to the number field case but works directly in any Dedekind domain.  Applied to $R = \Z_K$ with $B_j = \{0\}$ for all $j$, it gives a new proof of Theorem \ref{RVW2T}.  This new approach feels more transparent
and more fundamental, and we hope that it will be more amenable to further generalization.


\subsection{Applications of the Main Theorem}
\textbf{} \\ \\ \noindent
The generalization from polynomial congruences to polynomial congruences
with restricted outputs allows a wide range of applications.  As we mentioned in
\cite{Clark-Forrow-Schmitt14}, whenever one has a combinatorial existence theorem proved via the Schauz-Wilson-Brink Theorem (or an argument that can be viewed as a special case theroef) one can apply instead Theorem \ref{RVW2T} to get a lower bound on the number of solutions.  Moreover, most applications of
the Schauz-Wilson-Brink Theorem include a ``homogeneity'' condition which ensures the existence of a trivial solution.  Theorem \ref{RVW2T} applies also in the ``inhomogeneous case''.
\\ \indent
All of these applications can be generalized by allowing restricted outputs.  In  \cite{Clark-Forrow-Schmitt14} we gave three combinatorial applications of Theorem \ref{RVW2T}: to hypergraphs, to generalizations of the Erd\H os-Ginzburg-Ziv Theorem, and to weighted Davenport constants.  In the former two cases, we can (and shall) immediately apply the Main Theorem to get stronger results.  We include the proof of the hypergraph theorem to showcase the use of nonlinear polynomials.  We omit the proof
of the EGZ-type theorem: the proof given in \cite{Clark-Forrow-Schmitt14} of the special case adapts immediately.
\\ \\
Our Main Theorem leads to a generalization of the weighted Davenport constant that we call the \textbf{fat Davenport constant}.  This seems to be an interesting object of study in its own right
and we include some general discussion.  The fat Davenport constant can also be used to extend results of Alon-Friedland-Kalai on \emph{divisible subgraphs}.  This is a privileged application: the restricted output aspect
of the Main Theorem was directly inspired by \cite{Alon-Friedland-Kalai84}.
\\ \\
One reason that the combinatorial applications are interesting is that the upper bounds they give are -- in the unweighted, zero-output case --
accompanied by lower bounds coming from elementary combinatorial constructions,
which has the effect of showing sharpness in Schanuel's Theorem in certain cases.  It is an interesting challenge, not met here, to find other types of restricted input sets $A_i$ and restricted output sets $B_j$ illustrating sharpness in our generalized theorems.
\\ \\
Finally, we give an application of the Main Theorem to polynomial interpolation with fat targets.  As a special case we will deduce a generalization
of a Theorem of Troi-Zannier \cite{Troi-Zannier97} which was proved by them via more combinatorial means.  

\subsection{Acknowledgments}
\textbf{} \\ \\ \noindent
Thanks to Dino Lorenzini, Paul Pollack and Lori D. Watson for helpful discussions.  Thanks to Bob Rumely for suggesting the terminology of ``fat targets''.  I am deeply indebted to John R. Schmitt for introducing me to this rich circle of ideas, for many helpful remarks, and for Example \ref{SCHMITTEX}.

\section{Proof of the Main Theorem}


\subsection{A Generalized Alon-Friedman-Kalai Lemma}
\textbf{} \\ \\ \noindent
The following result is a generalization of \cite[Lemma A.3]{Alon-Friedland-Kalai84}.

\begin{lemma}
\label{GENERALSAFKLEMMA}
Let $R$ be a discrete valuation ring, with maximal
ideal $\pp = (\pi)$ and finite residue field $R/\pp \cong \F_q$.  Let
$v \in \Z^+$, and let $\mathcal{S}(v)$ be a set of coset
representatives for $\pp^v$ in $R$.  Let $T \subset R$ satisfy Condition (F):
no two distinct elements of $T$ are congruent modulo $\pp$, and let
$\overline{T}$ be the image of $T$ in $R/\pp^v$.   Let $x \in R$.  Put
\[ \mathbf{P}(x,v,T) = \prod_{y \in \mathcal{S}(v) \setminus \overline{T}}
(x-y) \]
and
\[ c(v) = \sum_{i=1}^{v-1} \left(q^i-1 \right). \]
Then we have:
\begin{equation}
\label{SAFKLEMMAEQ1}
\ord_{\pp} \mathbf{P}(x,v,T) \geq c(v),
\end{equation}
\begin{equation}
\label{SAFKLEMMAEQ2}
\ord_{\pp} \mathbf{P}(x,v,T) = c(v)
\iff \text{ there is } y \in \overline{T} \text{ such that } \ord_{\pp}(x-y) \geq v.
\end{equation}
\end{lemma}
\begin{proof}
Put $\mathbf{P}_0 = \prod_{y \in \mathcal{S}(v) \setminus \{ \pp^v\}} y$.  \\
Step 1: Suppose $\overline{T} = \{y_0\}$ and $\ord_{\pp}(x-y_0) \geq v$.  As $y$ runs through $\mathcal{S}(v) \setminus \overline{T}$, $x-y$ runs through a set of representatives of the nonzero cosets of $\pp^v$ in $R$, and since if $x \equiv y \not \equiv 0 \pmod{\pp^v}$ then $\ord_{\pp} x = \ord_{\pp} y$, we have
\[ \ord_{\pp} \mathbf{P}(x,v,T) = \ord_{\pp} \mathbf{P}_0 \] \[ = \sum_{i=0}^{v-1} i \cdot (\# \{x \in (\pp^i \cap \mathcal{S}(v)) \setminus (\pp^{i+1} \cap
\mathcal{S}(v)) \}) =
\sum_{i=0}^{v-1} i \cdot (q^{v-i}-q^{v-i-1}) \]
\[ = (q^{v-1}-q^{v-2}) + 2(q^{v-2}-q^{v-3}) + 3(q^{v-3}-q^{v-4}) + \ldots +
(v-1)(q-1)  \]
\[ = (q^{v-1} + q^{v-2} + \ldots + 1) - (v-1) = \sum_{i=1}^{v-1}(q^i-1) = c(v). \]
Step 2: Suppose $\overline{T} = \{y_0\}$ and $\ord_{\pp}(x-y_0) < v$.
 Then there is a unique $y_1 \in \mathcal{S}(v)$ with $x \equiv y_1 \pmod{\pp^v}$, and $y_1 \neq y_0$.  Then we have $\mathbf{P}(x,v,T) = \mathbf{P}_0 \left( \frac{x-y_1}{x-y_0} \right)$, so
\[ \ord_{\pp} \mathbf{P}(x,v,T) = c(v) + \ord_{\pp}(x-y_1) - \ord_{\pp}(x-y_0) > c(v). \]
Step 3: Suppose $\# \overline{T} > 1$.  Then $\mathbf{P}(x,v,T)$
is obtained from omitting factors from a product considered in Step 1 or Step 2.  Because no two elements of $T$ are congruent modulo $\pp$, the number of $y \in \overline{T}$
such that $\ord_{\pp}(x-y) \geq 1$ is either $0$ or $1$, and thus $\mathbf{P}(x,v,T)$ can be obtained from the product in Step 1 or Step 2 by omitting only factors of zero $\pp$-adic
valuation.  So $\ord_{\pp} \mathbf{P}(x,v,T) \geq c(v)$, and strict inequality holds
precisely when there is some $y \in \mathcal{S}(v) \setminus \overline{T}$ with
$\ord_{\pp}(x-y) \geq v$.
\end{proof}

\subsection{Alon-F\"uredi Over a Ring}
\textbf{} \\ \\ \noindent
The aim of this section is to prove the following result.

\begin{thm}(Alon-F\"uredi Over a Ring)
\label{STRONGALONFUREDI}
Let $R$ be a ring, let $A_1,\ldots,A_n \subset R$ satisfying Condition (D).  Put $\AA = \prod_{i=1}^n A_i$ and $a_i = \# A_i$ for all $1 \leq i \leq n$.  Let $P \in R[t] = R[t_1,\ldots,t_n]$.  Let
\[\mathcal{U}_{\AA} = \{x \in A \mid P(x) \neq 0\}, \ \mathfrak{u}_A = \# \mathcal{U}_{\AA}. \]
Then either $\uu_{\AA} = 0$ or $\uu_{\AA} \geq \mathfrak{m}(a_1,\ldots,a_n;a_1+\ldots+a_n - \deg P)$.
\end{thm}
\noindent
When $R$ is a field, this is the Alon-F\"uredi Theorem \cite[Thm. 4]{Alon-Furedi93}.   The key observation that the Combinatorial Nullstellensatz works over an arbitrary ring provided we impose Condition (D) is due to U. Schauz.  It was further developed in \cite[$\S 3$]{Clark14}.  The relevance of Condition (D) is shown in the following result.

\begin{thm}(\textbf{CATS Lemma}  \cite[Thm. 12]{Clark14})
\label{CATSLEMMA}
\label{CATS}
Let $R$ be a ring.  For $1 \leq i \leq n$, let $A_i \subset R$ be nonempty and finite.  Put $\AA = \prod_{i=1}^r A_i$.  For $1 \leq i \leq n$, let
$\varphi_i = \prod_{a_i \in A_i} (t_i-a_i)$. \\
a) (Schauz \cite{Schauz08a}) The following are equivalent: \\
(i) $X$ satisfies condition (D). \\
(ii) For all $f \in R[t_1,\ldots,t_n]$, if  $\deg_{t_i} f < \# A_i$ for all $1 \leq i\leq n$ and
$f(a) = 0$ for all $a \in \prod_{i=1}^n A_i$, then $f = 0$. \\
(iii) If $f|_{\AA} \equiv 0$, there are $g_1,\ldots,g_n \in R[t_1,\ldots,t_n]$ such that $f = \sum_{i=1}^n g_i \varphi_i$. \\
b) (Chevalley-Alon-Tarsi) The above conditions hold when $R$ is a domain.
\end{thm}
\noindent
With Theorem \ref{CATSLEMMA} in hand, Theorem \ref{STRONGALONFUREDI} can be established following the original argument of \cite{Alon-Furedi93}.  However, I find this argument a bit mysterious.
Theorem \ref{STRONGALONFUREDI} is the backbone of this work and a key barrier to further generalizations of Theorem
\ref{MAINTHM}.  Because of this I feel the need to give the most conceptually transparent argument possible.  For this we adapt a proof of Alon-F\"uredi due to Ball and Serra.

\begin{proof}
Step 1: We establish a variant of the Punctured Combinatorial Nullstellensatz
of Ball-Serra \cite[Thm. 4.1]{Ball-Serra09}.\footnote{The result established here
is obtained from the Punctured Combinatorial Nullstellensatz by (i) working
over an arbitrary ring under Condition (D) and (ii) neglecting multiplicities.} Let $R$ be a ring, let $A_1,\ldots,A_n \subset R$ satisfying
Condition (D), and put $\AA = \prod_{i=1}^n A_i$, $\YY = \prod_{i=1}^n Y_i$.  For $1 \leq i \leq n$ let
$\varnothing \neq Y_i \subset A_i$.  For $1 \leq i \leq n$, put
\[ \varphi_i(t) = \prod_{a_i \in A_i} (t_i-a_i), \ \psi_i(t) = \prod_{y_i \in Y_i}
(t_i-y_i). \]
Let $f \in R[\tt] = R[t_1,\ldots,t_n]$.  Suppose that for all $x \in \AA \setminus \YY$, $f(x) = 0$.
Then {\sc we claim} there are $g_1,\ldots,g_n,u \in R[\tt]$ such that
\[ f = \sum_{i=1}^n g_i \varphi_i + u \prod_{i=1}^n \frac{\varphi_i}{\psi_i}, \
\deg u \leq \deg f - \sum_{i=1}^n \left(\# A_i - \# Y_i \right).\]
{\sc proof of claim: } We perform polynomial division on $f$ by the
monic polynomial $\varphi_1$, then divide the remainder by the monic
polynomial $\varphi_2$, and so forth, finally dividing by $\varphi_n$
to get $f = \sum_{i=1}^n g_i \varphi_i + r$.  By \cite[$\S 3.1$]{Clark14}, we have $\deg r \leq \deg f$ and $\deg_{t_i} r < \deg \varphi_i$ for all $i$.
Dividing $r\psi_1$ by $\varphi_1$ we get
\[ r \psi_1 = r_1 \varphi_1 + s_1. \]
Then
\[ \deg_{t_1} s_1 < \deg \varphi_1 \]
whereas for all $i \neq 1$,
\[ \deg_{t_i} s_1 \leq \deg_{t_i} r \psi_1 = \deg_{t_i} r < \deg \varphi_i. \]
Since $s_1$ vanishes identically on $\AA$ and $\AA$ satisfies
Condition (D), Theorem \ref{CATSLEMMA} applies to show $s_1 = 0$: that is,
we may write $r = \frac{\varphi_1}{\psi_1} r_1$.  Continuing this process
with respect to $t_2,\ldots,t_n$, we get $r = \prod_{i=1}^n \frac{\varphi_i}{\psi_i} u$
with
\[ \deg u \leq \deg r - \sum_{i=1}^n (\deg(\varphi_i) - \deg(\psi_i)) \leq  \deg f - \sum_{i=1}^n \left(\# A_i - \# Y_i \right).  \]
Step 2: Put $A = \prod_{i=1}^n A_i$, and let $f \in R[t_1,\ldots,t_n]$.  We may
assume that $f$ does not vanish identically on $\AA$. We go by induction on $n$, the case $n = 1$ following from Theorem \ref{CATSLEMMA}.  Suppose $n \geq 2$ and the
result holds for $n -1$.  Define
\[ Y_i = \begin{cases} A_i, & 1 \leq i \leq n-1 \\
\{y \in X_n \mid f(t_1,\ldots,t_{n-1},y) \neq 0 \} & i = n \end{cases}. \]
By our hypothesis on $f$, $Y_n \neq \varnothing$.  Let
$y \in Y_n$.  We apply Step 1 to $f$, getting
\[ f = \sum_{i=1}^n g_i \varphi_i + u \frac{\varphi_n}{\psi_n} \]
and put $w(t_1,\ldots,t_{n-1}) = u(t_1,\ldots,t_{n-1},y)$.  Then \[\deg w \leq \deg u \leq \deg f - \# A_n + \# Y_n, \] and for all $x' = (x_1,\ldots,x_{n-1}) \in \prod_{i=1}^{n-1} A_i$, we have $f(x',y) = 0 \iff w(x') = 0$.  By induction
there are $a_1,\ldots,a_{n-1} \in \Z^+$ with $1 \leq a_i \leq \# A_i$ for all $i$
and \[\sum_{i=1}^{n-1} a_i = \left(\sum_{i=1}^{n-1} \# A_i\right) - \deg w \geq \left(\sum_{i=1}^{n-1} \# A_i\right) - \deg u\] such
that $w$ is nonvanishing at at least $\prod_{i=1}^{n-1} a_i$ points of $\prod_{i=1}^{n-1} A_i$.  The $a_1,\ldots,a_{n-1}$ depend on $y$,
but if we choose $a_1,\ldots,a_{n-1}$ so as to minimize $\prod_{i=1}^{n-1} a_i$, then we find $(\prod_{i=1}^{n-1} a_i)(\# Y_n)$
points of $X$ at which $f$ is nonvanishing, hence
\[ \mathcal{U}_{\AA} \geq \mm(\# A_1,\ldots,\# A_n; \sum_{i=1}^{n-1} a_i + \# Y_n). \]
Since
\[ \sum_{i=1}^{n-1} a_i + \# Y_n \geq
 (\sum_{i=1}^{n-1} \# A_i) - \deg u + (\deg u + \# A_n-\deg f) =
\sum_{i=1}^n \# A_i - \deg f, \]
we have $\mathcal{U}_{\AA} \geq
\mm(\# A_1,\ldots,\# A_n; \sum_{i=1}^n \# A_i - \deg f)$.
\end{proof}

\begin{remark}
\label{AFSHARP}
Theorem \ref{STRONGALONFUREDI} is \emph{sharp} in the following sense: let
$R$ be a ring, $A_1,\ldots,A_n \subset R$ satisfying Condition (D), and put
$\AA = \prod_{i=1}^n A_i$.  Let $d \in \Z^+$.  There is a degree $d$ polynomial $f \in R[t_1,\ldots,t_n]$ which is nonzero at precisely $\mm(\# A_1,\ldots,\# A_n; \sum_{i=1}^n \# A_i - d)$ points of $\AA$.  In fact something stronger holds: let $y = (y_1,\ldots,y_n) \in \Z^n$ with $1 \leq y_i \leq \# A_i$ for all $i$.  For $1 \leq i \leq n$, choose $Y_i \subset A_i$ with $\# Y_i = y_i$, and put $f = \prod_{i=1}^n \prod_{x \in Y_i} (t_i-x)$.  Then $\deg f = \sum_{i=1}^n (\# A_i - \# Y_i)$ and $f$ is nonvanishing precisely on
$\prod_{i=1}^n (A_i \setminus Y_i)$, a subset of size $\prod_{i=1}^n y_i$.
\end{remark}

\begin{remark}
Both of the main results of \cite{Ball-Serra09} -- namely Theorems 3.1 and 4.1 -- can be generalized by replacing the arbitrary field $\mathbb{F}$ by an arbitrary
ring $R$ under the assumption that the sets satisfy Condition (D).  In the former case the argument adapts immediately; in the latter case it requires some mild modifications.
\end{remark}

\subsection{Proof of the Main Theorem}

\begin{proof}
We will prove Theorem \ref{MAINTHM}.  As in Remark \ref{LOCALIZATIONREMARK}, we may assume $R$ is a DVR, and thus our assumption on $A_1,\ldots,A_n,B_1,\ldots,B_r$ becomes Condition (F).  Let $\AA = \prod_{i=1}^n A_i$.  For $1 \leq j \leq r$, let $\overline{B_j}$ be the image of $B_j$ in $R/\pp^{v_j}$.  For $a \in \Z^+$, let $\mathcal{S}(a)$ be a set of coset representatives for $\pp^a$ in $R$.  Put
\[ Q(t) = \prod_{j=1}^r \prod_{y \in \mathcal{S}(v_j) \setminus \overline{B_j}} (P_j(t) - y) \in R[t]. \]
For $1 \leq j \leq s$ put
\[ c_j = \sum_{i=1}^{v_j-1} (q^i-1), \]
and put
\[  c = \sum_{j=1}^r c_j. \]
Let $\overline{R} = R/\pp^{c+1}$.  Let $\overline{Q}$ be the image of $Q$ in $\overline{R}$ and $\overline{\AA}$ the image of $\AA$ in $\overline{R}^n$.  Then \[\deg \overline{Q} \leq \deg Q = \sum_{j=1}^r (q^{v_j}-b_j)\deg f_j. \]
Because of Condition (F), the natural map $\AA \mapsto \overline{\AA}$ is a bijection.  Let
\[\mathcal{U} = \{ \overline{x} \in \overline{\AA} \mid
\overline{Q}(\overline{x}) \neq 0\}. \]
Let $\overline{x} \in \overline{\AA}$.  Using Lemma (\ref{GENERALSAFKLEMMA}), we get
\[ \overline{x} \in \mathcal{U}  \iff \overline{Q}(\overline{x}) \neq 0 \] \[ \iff \ord_{\pp}( Q(x)) \leq c
 \stackrel{(\ref{SAFKLEMMAEQ1})}{\iff}  \forall 1 \leq j \leq r, \ \ord_{\pp} \prod_{y \in \mathcal{S}(v) \setminus \overline{B_j}} (f_j(x)-y) \leq c_j \]
\[ \stackrel{(\ref{SAFKLEMMAEQ2})}{\iff} \forall 1 \leq j \leq r, \exists b_j \in \overline{B_j} \text{ such that } \ord_{\pp}(f_j(x) - b_j) \geq v_j \] \[ \iff x \in Z_{\AA}^{\BB}. \]
Thus $\# \mathcal{U} = \zz_{\AA}^{\BB}$.
Applying Theorem \ref{STRONGALONFUREDI} to $\overline{R}$, $\overline{Q}$ and $\overline{A}$, we get that $\# Z_{\AA}^{\BB}= 0$
\[ \# Z_{\AA}^{\BB} \geq
\mm(\# A_1,\ldots,\# A_n;\sum_{i=1}^n \# A_i - \deg \overline{Q})  \] \[\geq \mm(\# A_1,\ldots,\# A_n;\sum_{i=1}^n \# A_i - \sum_{j=1}^r (q^{v_j} -\# B_j) \deg(f_j)). \qedhere \]
\end{proof}


\section{Applications}

\subsection{Hypergraphs}
\textbf{} \\ \\ \noindent
\newcommand{\FF}{\mathcal{F}}
A \textbf{hypergraph} is a finite sequence $\FF = (\FF_1,\ldots,\FF_n)$
of finite subsets of some fixed set $X$.  We say that $n$ is the \textbf{length} of $\FF$.  The \textbf{maximal degree} of $\FF$ is $\max_{x \in X} \# \{ 1 \leq i \leq n \mid x \in \FF_i\}$.  For $m \in \Z^+$ and $\varnothing \neq B \subset  \Z/m\Z$,
\[ N_{\FF}(m,B) = \# \{ J \subset \{1,\ldots,n\} \mid \# (\bigcup_{i \in J} \FF_i) \in B \pmod{m} \}, \]
and for $n,d \in \Z^+$, let
\[ \mathcal{N}_{n,d}(m) = \min N_{\FF}(m,0), \]
the minimum ranging over set systems of length $n$ and maximal degree at most $d$.  Let $f_d(m)$ be the least $n \in \Z^+$ such that for any degree $d$ set system $\FF$ of length $n$, there is a nonempty subset $J \subset \{1,\ldots,n\}$ such that $m \mid \# (\bigcup_{i \in J} \FF_i)$.   Thus
\begin{equation}
\label{AKLMRSEQ1}
 f_d(m) = \min \{n \in \Z^+ \mid \mathcal{N}_{n,d}(m) \geq 2 \}.
\end{equation}


\begin{thm}
\label{AKLMRS}
Let $p$ be a prime, and let $\varnothing \subset B \subset \Z/p^v \Z$
be a subset, no two distinct elements of which are congruent modulo $p$.  Let $d,n \in \Z^+$, and let $\FF = (\FF_1,\ldots,\FF_n)$
be a hypergraph of maximal degree at most $d$.  Then: \\
a) $\mathcal{N}_{\FF}(p^v,B)$ is either $0$ or at least $2^{n-d(p^v-\# B)}$.\\
b) If $0 \in B$ and $n > d(p^v-\# B)$, then there is $\varnothing \neq J \subset
\{1,\ldots,n\}$ such that $p^v \mid \# \bigcup_{i \in J} \FF_i$.
\end{thm}
\begin{proof}
Put
\[ h(t_1,\ldots,t_n) = \sum_{\varnothing \neq J \subset \{1,\ldots,n\}}
(-1)^{\# J + 1}  ( \# \bigcap_{j \in J} \FF_i) \prod_{j \in J} t_j. \]
Then $\deg h \leq d$ and $h(0) = 0$.  For any $x \in \{0,1\}^n$, let
$J_x = \{ 1 \leq j \leq n \mid x_j = 1\}$.  The Inclusion-Exclusion Principle implies
\[ h(x) = \# \bigcup_{j \in J_x} \FF_j, \]
so $\mathcal{N}_{\FF}(p^v,B) = \{x \in \{0,1\}^n \mid h(x) \in B \pmod{p^v} \}$.
Applying Theorem \ref{MAINCORZ}a) establishes part a), and applying
Theorem \ref{MAINCORZ}b) establishes part b).
\end{proof}
\noindent
When $B = \{0\}$, Theorem \ref{AKLMRS}a) is \cite[Thm. 4.8]{Clark-Forrow-Schmitt14}a) and Theorem \ref{AKLMRS}b) gives the upper bound in a result of Alon-Kleitman-Lipton-Meshulam-Rabin-Spencer \cite[Thm. 1]{AKLMRS91}.  They also showed that $f_d(m) \geq d(m-1)+1$, so Theorem \ref{AKLMRS}b) is sharp when $\# B = 1$.  The following example extends this construction and implies that Theorem \ref{AKLMRS}b) is
sharp for all $d,\# B \in \Z^+$.

\begin{example}(J.R. Schmitt)
\label{SCHMITTEX}
Let $b,d \in \Z^+$.  Choose $m,a \in \Z^+$ with $m > b$ and $\gcd(a,m) = 1$.  Let $\{A_{i,j}\}_{1 \leq i \leq m-b, \ 1 \leq j \leq d}$ be pairwise disjoint sets, each of cardinality $m$.  Let $\{V_i\}_{1 \leq i \leq m-b}$
be disjoint sets, each of cardinality $a$ and disjoint from all the $A_{ij}$.  Put \[B=\{m, m-a, m-2a, \ldots, m-(b-1)a\} \subset \Z/m\Z, \] and \[\mathcal{F} = \{A_{ij} \cup V_i\}_{1 \leq i \leq m-b,1 \leq j \leq d}. \] Then $\# {\mathcal{F}} = d(m-b)$ but $ \# \bigcup_{F \in {\mathcal{F}}_0} \notin B\pmod{m}$ for any $\emptyset \neq {\mathcal{F}}_0 \subset {\mathcal{F}}$.

\end{example}


\subsection{Fat Davenport Constants}
\textbf{} \\ \\ \noindent
Let $(G,+)$ be a nontrivial finite commutative group.  The \textbf{Davenport constant} $D(G)$
is the least number $n$ such that for any sequence $\{g_i\}_{i =1}^n$ in $G$,
there is a nonempty subset $J \subset \{1,\ldots,n\}$ such that $\sum_{i \in J} g_i = 0$.  There are unique integers $1 < n_1 \mid n_2 \ldots \mid n_r$ such that
\[ G \cong \bigoplus_{i=1}^r \Z/n_i \Z; \]
let us call $r$ the \textbf{rank} of $G$.  The pigeonhole principle implies
\[ D(G) \leq \# G. \]
Let $e_i \in \bigoplus_{i=1}^r \Z/n_i \Z$ be the element with $i$th coordinate
$1$ and all other coordinates zero.  Then the sequence
\[ \overbrace{e_1,\ldots,e_1}^{n_1-1},\overbrace{e_2,\ldots,e_2}^{n_2-1},
\ldots,\overbrace{e_r,\ldots,e_r}^{n_r-1} \]
shows that
\begin{equation}
\label{DAVENPORTEQ2}
D(G) \geq 1 + \sum_{i=1}^r (n_i-1) =: d(G).
\end{equation}
It is now clear that $d(G) = D(G)$ when $G$ has rank $1$ (i.e., is cyclic).  Olson
showed that this equality also holds when $G$ has rank $2$ and when $G$ is a $p$-group of arbitrary rank.  There are infinitely groups of rank $4$ with $d(G) < d(G)$.  Whether $d(G) = D(G)$ for all groups of rank $3$, or for all
groups $G$ with $n_1 = \ldots = n_r$, are major open questions.  Olson used group ring methods (which, by the way, are used to prove the best known upper bound for $D(G)$ for a general group $G$, see \cite{AGP94}) to prove $d(G) = D(G)$ for $p$-groups, but
in fact for a $p$-group $G$ the Davenport constant can be expressed in terms of
systems of congruences modulo powers of $p$ with solutions in $A_i = \{0,1\}$.  This
was first observed by Schanuel \cite{Schanuel74}.
\\ \\
Let $G$ be a finite commutative group of exponent $e$, and let $\AA = \{A_i\}_{i=1}^{\infty}$ be a sequence of finite, nonempty subsets of $\Z$.  Then given a sequence $\underline{g} = \{g_i\}_{i=1}^{n}$ in $G$ we may associate an \textbf{\AA-weighted subsequence} $\{a_i g_i\}_{i=1}^{n}$ by
selecting $a_i \in A_i$.  We say an $\AA$-weighted subsequence is \textbf{empty} if $a_i = 0 \in A_i$ for each $i$.
\\ \\
When each $A_i$ contains $0$ and at least one other element not divisible by $\exp G$, we define the \textbf{weighted Davenport constant} $D_{\AA}(G)$ as the least $n$ such that every sequence $\{g_i\}_{i=1}^n$ in $G$
has a nonempty $\AA$-weighted zero-sum subsequence: i.e., there are $a_1 \in A_1,\ldots,a_n \in A_n$, not all $0$, such that $\sum_{i=1}^n a_i g_i = 0$.  As in the classical case, an immediate pigeonhole argument shows
\[ D_{\AA}(G) \leq \# G. \]


\begin{thm}
\label{TZB}
(Troi-Zannier \cite{Troi-Zannier97}, Brink \cite{Brink11})
Let $G \cong \bigoplus_{i=1}^r \Z/p^{v_i}$ be a $p$-group of exponent $p^v$.  Let $n \in \Z^+$, and let $\AA = (A_1,\ldots,A_n)$
with each $A_i \subset \Z/p^v\Z$ nonempty and such that no two distinct elements are congruent modulo $p$.  If
\[  \sum_{i=1}^n (\# A_i-1) > \sum_{j=1}^r (p^{v_i}-1) = d(G)-1, \]
every sequence of length $n$ in $G$ has a nonempty $\AA$-weighted zero-sum subsequence.
\end{thm}
\noindent
Troi-Zannier's proof uses group ring methods.  They remark on their inability to push through a Chevalley-Warning style proof in the general case; this is what Brink does using the Schauz-Wilson-Brink Theorem.
\\ \\
When $A_1 = \ldots = A_n = A$, we write $D_{A}(G)$ for $D_{(A_1,\ldots,A_n)}(G)$.  Most of the study of weighted Davenport
constants has been devoted to this case.  In this case, Theorem \ref{TZB} becomes the upper bound
\begin{equation}
 \label{TZBEQ}
 \text{For all $p$-groups $G$, }
  D_A(G) \leq \bigg{\lceil} \frac{d(G)}{\# A-1} \bigg{\rceil}.
  \end{equation}
Thangadurai gives some evaluations of
$D_A(G)$ when $G = \Z/p\Z$ using elementary methods \cite[Thm. 2]{Thangadurai07} and when $G = \bigoplus_{i=1}^r \Z/p\Z$ using (\ref{TZBEQ}) \cite[Cor. 1.2]{Thangadurai07}. \\ \indent The case $A = \{-1,0,1\}$ is especially natural; $D_{\{-1,0,1\}}(G)$ is called
the \textbf{plus-minus weighted Davenport constant} and denoted $D_{\pm}(G)$.  Its study predates the general weighted case \cite{Stein77}, \cite{Mead-Narkiewicz82} and, in some form, the unweighted Davenport constant \cite{Shannon56}.  However, in this case upper bounds of the form (\ref{TZBEQ}) are quite far from the truth.  Some reflection on the uniqueness of binary expansions shows
\[ D_{\pm}(\Z/n\Z) = \lfloor \log_2 n \rfloor + 1, \]
while for any finite group $G$, an elementary 
argument \cite[Thm. 4.1]{DAGS12} gives 
\[ D_{\pm}(G) \leq \lfloor \log_2 \# G \rfloor + 1. \]
For recent work giving lower bounds and some equalities 
for $D_{\pm}(G)$ see \cite{MOS14}. \\ \indent
Thus certain choices of $A \subset \Z$ lead to a behavior of $D_A(G)$ which is much different from the extremal case (attained for $A = \{0,1\}$).  It would be interesting to further understand this phenomenon.
\\ \\
Here is a further generalization of the Davenport constant.   We give ourselves: \\
$\bullet$ A finite group $G = \bigoplus_{i=1}^r \Z/n_i \Z$ with $1 < n_1 \mid \ldots \mid n_r$ (so $G$ has exponent $n_r$);\\
$\bullet$ A sequence $\AA = \{A_i\}_{i=1}^{\infty}$ of nonempty finite subsets of $\Z$.  \\
$\bullet$ A nonempty subset $B \subset G$.
\\ \\
For a sequence $\underline{g} = \{g_i\}_{i=1}^n$ in $G$, let \[N_{\AA}^B(\underline{g}) = \# \{a \in A \mid \sum_{i=1}^n a_i g_i \in B\}, \]
i.e., the number of $\AA$-weighted subsequences of $\underline{g}$ with sum in $B$.  For $b \in B$, we put 
$N_{\AA}^b(G) = N_{\AA}^{\{b\}}(G)$.  Also put
\[ N_{\AA}^B(G,n) = \min_{\underline{g} \in G^n \mid \sum(\underline{g}) \cap B \neq \varnothing}  N_{\AA}^B(\underline{g}), \]
that is, we range over all sequences of length $n$ in $G$ and take the least positive number of $\AA$-weighted subsequences with sum in $B$.
\\ \\
If every $A_i$ contains $0$ and also at least one element not divisible by
$\exp G$, and $0 \in B$ then  we
define the \textbf{fat Davenport constant} $D_{\AA}^B(G)$ to be the least $n \in \Z^+$ such that every length $n$ sequence in $G$ has a nonzero $\AA$-weighted subsequence with sum in $B$.  We have
\[ D_{\AA}^B(G) \leq D_{\AA}(G) \leq D(G). \]
We write $D^B(G)$ for $D_{ \{0,1\}}^B$.  Evidently $D^B(G) \leq D(G)$.  It would be interesting to give upper bounds on $N^B(\underline{g})$ depending only on $\# B$ and the length of $\underline{g}$.  
\\ \\
For a sequence $\underline{g}$, let \[\Sigma (\underline{g}) = \left\{ \sum_{i \in J} g_i \mid J \subset \{1,\ldots,n\} \right\} \]
be the set of all subsequence sums of $\underline{g}$.
\\ \\
For a general group $G$ we have little insight into the quantities $D_{\AA}^B(G)$ and $N_{\AA}^B(G,n)$, and we will content ourselves here with a few observations.

\begin{thm}
\label{NGTHM} \textbf{} \\
Let $(G,+)$ be a finite commutative group, and let $\underline{g} = \{g_i\}_{i=1}^n$ be a sequence in $G$. \\
a) (\cite[Thm. 2]{Olson69b}) We have $N^0(\underline{g}) = \max \{1,2^{n+1-D(G)} \}$.  \\
b) (\cite[Thm. 2]{CCQWZ11}) For all $x \in \Sigma(\underline{g})$, we have $N^x(\underline{g}) \geq  2^{n+1-D(G)}$. \\
c) (\cite[Prop. 4]{CCQWZ11}) If for some $y \in G$ we have $N^y(\underline{g}) =
2^{n+1-D(G)}$, then $N^x(\underline{g}) \geq 2^{n+1-D(G)}$ for all $x \in G$.
\end{thm}

\begin{cor}
\label{NGCOR}
Let $\underline{g}$ be a sequence of length $n$ in $G$, and let $\{0\}
\subsetneq B \subset G$.  Then: \\
a) We have $N^B(\underline{g}) \geq (\# \sum(\underline{g}) \cap B) \cdot 2^{n+1-D(G)}$. \\
b) We have that $N^B(\underline{g})$ is $0$ or is at least $2^{n+1-D(G)}+1$.
\end{cor}
\begin{proof}
a) By Theorem \ref{NGTHM}b),
$b \in B$ which occurs as a subsequential sum of $\underline{g}$ must occur
at least $2^{n+1-D(G)}$ times.  \\
b) We have $N^B(\underline{g}) = 0$ iff $\sum(\underline{g}) \cap B = \varnothing$.  We may assume this is not the case: there is $y \in \sum(\underline{g}) \cap B$, and then part a) gives $N^B(\underline{g}) \geq 2^{n+1-D(G)}$.  Certainly
$N^B(\underline{g}) \geq N^y(\underline{g})$, and by
Theorem \ref{NGTHM}c), if
$N^y(\underline{g}) = 2^{n+1-D(G)}$, then \[N^B(\underline{g}) \geq (\# B)2^{n+1-D(G)} > 2^{n+1-D(G)}. \qedhere \]
\end{proof}

\begin{remark}
Suppose $0 \in B$
and $B$ is a large subset of $G$.  When $\sum(\underline{g}) \cap B$ is
large, Corollary \ref{NGCOR}a) gives a good lower bound on $N^B(\underline{g})$.  When $\sum(\underline{g}) \cap B$ is small, then
there ought to be significantly more than $2^{n+1-D(G)}$ zero-sum subsequences.
\end{remark}

\begin{remark}
If $B$ is a subgroup of $G$, then $D_{\AA}^B(G) = D_{\AA}(G/B)$.
\end{remark}
\noindent
However, when $G$ is a $p$-group, our Main Theorem can be applied.

\begin{thm}
\label{FATDAVENPORT}
Let $p$ be a prime; let $1 \leq v_1 \leq \ldots \leq v_r$, and let $G = \bigoplus_{j=1}^r \Z/p^{v_j} \Z$.  Let $\{A_i\}_{i=1}^{\infty}$ be a sequence
of subsets of $\Z/p^{v_r}\Z$, and for $1 \leq j \leq r$ let
$B_j \subset \Z/p^{v_j} \Z$.  Suppose each $A_i$ and $B_j$ is nonempty
and has no two distinct elements congruent modulo $p$.  Let $B = \prod_{j=1}^r B_j$.
Let $\underline{g} = \{g_i\}_{i=1}^n$ be a sequence in $G$. \\
a) The number of $\AA$-weighted subsequences of $\underline{g}$ with $\sum_{i=1}^n a_i g_i \in B$
is $0$ or at least
\[ \mm\left(\# A_1,\ldots,\# A_n;\sum_{i=1}^n \# A_i - \sum_{j=1}^r (p^{v_j} - \# B_j)\right). \]
b) If $0$ lies in each $A_i$ and $B_j$, then there is a nonempty $\AA$-weighted subsequence of $\underline{g}$ with sum $\sum_{i=1}^n a_i g_i \in B$
if
\[ \sum_{i=1}^n \left(\# A_i -1\right) > \sum_{j=1}^r (p^{v_j} - \# B_j). \]
\end{thm}
\begin{proof}
For $1 \leq i \leq n$ and $1 \leq j \leq r$, let $g_i = (a_1^{(i)},\ldots,a_r^{(i)})$ and $f_j(t_1,\ldots,t_n) = \sum_{i=1}^n a_j^{(i)} t_i$.  Apply Theorem \ref{MAINTHMZ}.
\end{proof}

\begin{remark}
Taking $B = \{0\}$ gives \cite[Thm. 4.6]{Clark-Forrow-Schmitt14}.  The latter implies Corollary \ref{TZB}, which implies Olson's Theorem that $D(G) = d(G)$ for $p$-groups.
\end{remark}
\noindent
The following result is the generalization of Theorem \cite[Thm. 4.11]{Clark-Forrow-Schmitt14} obtained by applying the Main Theorem.  The proof carries over immediately and is omitted.

\begin{thm}
\label{EGZTHM}
Let $k,r$, $v_1 \leq \ldots \leq v_r$ be positive integers, and let
$G = \bigoplus_{i=1}^r \Z/p^{v_i} \Z$.  Let $A_1,\ldots,A_n,B_1,\ldots,B_r$ be nonempty subsets of $\Z$, each nonempty with distinct elements pairwise incongruent modulo $p$ and with $0 \in A_i$ for all $i$.  Put
\[ A = \prod_{i=1}^n A_i, \ a_M = \max \# A_i. \]
For $x \in G$, let  $\operatorname{EGZ}_{A,k}(B)$
be the number of $(a_1,\ldots,a_n) \in A$ such that $a_1 x_1 + \ldots + a_n x_n \in \prod_{j=1}^r B_j$ and $p^k \mid \# \{ 1 \leq i \leq n \mid a_i \neq 0\}$.  Then either
$\operatorname{EGZ}_{A,k}(B) = 0$ or
\[\operatorname{EGZ}_{A,k}(B) \geq \mm(\# A_1,\ldots,\# A_n; \# A_1 + \ldots + \# A_n - \sum_{j=1}^r (p^{v_j}-\# B_j) - (a_M-1)(p^k-1)). \]

\end{thm}

\subsection{Divisible Subgraphs}
\textbf{} \\ \\
Here, a \textbf{graph} is a relation $\sim$ -- called \emph{incidence} -- between two finite sets $V$ and $E$ such that every $e \in E$ is incident to exactly two elements of $V$.  If $\# V = r$ we will identify $V$ with $\{1,\ldots,r\}$.  A subgraph
is induced by restricting the incidence relation to a subset $E' \subset E$.  We say
a graph is empty if $E = \varnothing$.  For $q \in \Z^+$, a graph $\GG = (V(\GG),E(\GG))$ is \textbf{q-divisible} if for all $x \in V(G)$, $q \mid \deg x$ \cite{Alon-Friedland-Kalai84}.  An empty graph is $q$-divisible for all $q$.  We say a
graph is \textbf{q-atomic} if it admits no nonempty $q$-divisible subgraph.
\\ \\
For $r \geq 2$ and $q \in \Z^+$, let $E(r,q)$ be the least $n \in \Z^+$ such that
every graph with $r$ vertices and $n$ edges admits a nonempty $q$-divisible subgraph.
we have $E(2,q) = q$ for all $q$; henceforth we suppose $r \geq 3$.

\begin{thm}(\cite{Alon-Friedland-Kalai84})
\label{AFK2.2}
For $r \geq 3$ and $q \in \Z^+$, we define
\[ \mathcal{E}(r,q) := \begin{cases} (q-1)r + 1 & \text{$q$ odd} \\ (q-1)r - \frac{q}{2} +1 & \text{$q$  even} \end{cases}. \]
a) We have $\mathcal{E}(r,q) \leq E(r,q)$. \\
b) We have $\mathcal{E}(r,q) = E(r,q)$ if $q$ is a prime power.
\end{thm}
\noindent
The proof of part a) is by a simple direct construction of $q$-atomic graphs which we do not revisit here.  The proof of part b)
is by connection with the Davenport constant.  Or at least essentially: the term ``Davenport constant'' does not appear in \cite{Alon-Friedland-Kalai84}.  By making this connection explicit we can slightly sharpen their results.

\begin{thm}
\label{BIGAFK}
For $r \geq 3$, $q \in \Z^+$, we define
\[G(r,q) = \begin{cases} \bigoplus_{i=1}^r \Z/q\Z & q \text{ odd} \\
\bigoplus_{i=1}^{r-1} \Z/q\Z \oplus \Z/\frac{q}{2}\Z & q \text{ even} \end{cases} \]
and
\[ D(r,q) = D(G(r,q)). \]
a) We have
\begin{equation}
\label{DAVAFKEQ}
 d(G(r,q)) = \mathcal{E}(r,q) \leq E(r,q) \leq D(r,q).
\end{equation}
b) A graph with $r$ vertices and $n$ edges has at least
$2^{n+1-D(r,q)}$ $q$-divisible subgraphs. \\
c) \cite[Thm. 3.5]{Alon-Friedland-Kalai84} If $q$ is a prime power, then $E(r,q) = \mathcal{E}(r,q)$
and a graph with $r$ vertices and $n$ edges has at least $2^{n+1-\mathcal{E}(r,q)}$ $q$-divisible subgraphs.
\end{thm}
\begin{proof}
a) The equality $d(G(r,q)) = \mathcal{E}(r,q)$ is immediate, and $\mathcal{E}(r,q) \leq E(r,q)$ is Theorem \ref{AFK2.2}a).  Let $\mathcal{G}$ be a graph with $r$ vertices and $n$ edges, and let
$A = (a_j^{(i)})_{1 \leq i \leq n,\ 1 \leq j \leq r }$ be its incidence matrix.
Put $\Z[\tt] = \Z[t_1,\ldots,t_n]$.  Then
\[I \subset \{1,\ldots,e\} \mapsto x^I \in \{0,1\}^e, \ x^I_i = \begin{cases}
1 & i \in I \\ 0 & i \notin I \end{cases} \]
induces a bijection between the $q$-divisible subgraphs of $\mathcal{G}$
and the solutions $x \in \{0,1\}^n$ to the system of linear congruences
\[ \forall 1 \leq j \leq r, \ \sum_{i \in I} t_j a_j^{(i)} \equiv 0 \pmod{q}\]
and thus to zero-sum subsequences of $\underline{a} = \{a^{(i)}\}_{i=1}^n$
in  $\bigoplus_{j=1}^r \Z/q\Z$.  Thus
\[ E(r,q) \leq D(\bigoplus_{j=1}^r \Z/q\Z). \]
When $q$ is odd, $D(r,q) = D(G(r,q))$.   When $q$ is even, the fact that every edge is incident to precisely two vertices can be exploited to improve the bound:
\begin{equation}
\label{AFK2.2EQ}
\forall 1 \leq i \leq e, \ \sum_{j=1}^n a_{j}^{(i)} = 2 \equiv 0 \pmod{q}.
\end{equation}
In group-theoretic terms, (\ref{AFK2.2EQ}) means that the terms of $\underline{a}$ lie in the subgroup
\[ G' = \{(x_1,\ldots,x_n) \in \bigoplus_{i=1}^n \Z/q \Z \mid \sum_j x_j \equiv 0 \pmod{2} \} \cong G(r,q). \]
Thus again we find $E(r,q) \leq D(r,q)$.  \\
b) We have seen that $q$-divisible subgraphs correspond bijectively
to zero-sum subsequences of a sequence $\underline{a}$ in a group isomorphic
to $G(r,q)$.  Apply Theorem \ref{NGTHM}b).  \\
c) Since $q$ is a prime power, $G(r,q)$ is a $p$-group and thus $D(G(r,q)) = d(G(r,q)) = \mathcal{E}(r,q)$.  The result now follows from parts a) and b).
\end{proof}

\begin{remark}
Alon-Friedland-Kalai conjecture that $E(r,q) \leq (q-1)r + 1$ for all $q \in \Z^+$ \cite[Conj. 3.7]{Alon-Friedland-Kalai84}.  This would follow if $d(G) = D(G)$ for all direct sums of copies of one finite cyclic group.  As mentioned above, this is an important open problem.  When $q$ is odd, this conjecture is equivalent to $E(r,q) = \mathcal{E}(r,q)$; when $q$ is even it gives
\[ (q-1)r-\frac{q}{2} + 1 \leq E(r,q) \leq (q-1)r + 1. \]
Again $E(r,q) = \mathcal{E}(r,q)$ would follow
from $d(G(r,q)) = D(G(r,q))$.  To the best of my knowledge, whether this equality
holds for all even $q$ is also open.
\end{remark}

\begin{remark}
 We have allowed graphs with multiple edges, and in fact the graphs used in the proof of Theorem \ref{AFK2.2}a) have multiple edges.  We have not allowed loops, but we could have, as we now discuss.  There are two possible conventions on how loops contribute
to the incidence matrix (equivalently, the degree of a vertex). \\ \indent $\bullet$ If we take the
\textbf{topologist's convention} that placing a loop at a vertex increases its degree by $2$, then Theorem \ref{BIGAFK} holds verbatim for graphs with loops. \\ \indent  $\bullet$ If we take the \textbf{algebraist's convention} that placing a loop at a vertex increases its degree by $1$, then the parity phenomenon of
(\ref{AFK2.2EQ}) is lost, and for even $q$ as well as odd we get $E(r,q) \leq
D( \bigoplus_{i=1}^r \Z/q\Z)$.  In this case, the graph with $q-1$ loops placed at every vertex is $q$-atomic and shows \[d(\bigoplus_{i=1}^r \Z/q\Z) = (q-1)r + 1 \leq E(r,q). \]  When $q$ is a prime power we get $E(r,q) = (q-1)r + 1$
whether $q$ is even or odd.
\end{remark}
\noindent
The connection with Davenport constants motivates us to explore a more general graph-theoretic setup.  We first present a generalization which gives a graph-theoretic interpretation to the Davenport constant of any finite commutative group.  The proofs are quite similar to those given above and are left to the reader.
\\ \\
Let $\qq = (q_1,\ldots,q_r) \in (\Z^+)^r$ with
$1 < q_1 \mid q_2 \mid \ldots q_r$ and put
\[G(\qq) = \bigoplus_{i=1}^r \Z/q_i \Z. \]
When $q_1$ is even, there is a surjective group homomorphism
\[ \Phi: G(\qq) \ra \Z/2\Z, \ (x^{(1)},\ldots,x^{(r)}) \mapsto \sum_{j=1}^r x^{(j)} \pmod{2}. \]
Thus $G'(\qq) := \Ker \Phi$
is an index $2$ subgroup of $G(\qq)$.  In this case we set $\qq' = (\frac{q_1}{2},q_2,\ldots,q_r)$.

\begin{lemma}
\label{AFK2.6}
If $q_1$ is even, then
\[ G'(\qq) \cong G(\qq'). \]
\end{lemma}
\noindent
If $q_1$ is odd, we put $G'(\qq) = G(\qq)$.
\\ \\
Let $\mathcal{G} = (V,E)$ be a finite graph with $V = \{1,\ldots,r\}$.  A subgraph $\mathcal{G}' = (V,E')$ is $\qq$-\textbf{divisible} if for all $1 \leq j \leq r$, $q_j \mid \deg j$.  More generally, for $g = (g^{(j)})_{j=1}^r \in G(\qq)$, a subgraph $\mathcal{G}'$ is \textbf{of type} $(\qq,g)$ if for all $1 \leq j \leq r$ we have
\[ \deg j \equiv g^{(j)} \pmod{q_j}. \]
We then get the following generalization of Theorem \ref{BIGAFK}.


\begin{thm}
\label{AFK2.7}
Let $\qq \in (\Z^+)^n$, and let $\mathcal{G}$ be a finite graph with vertex set
$V = \{1,\ldots,r\}$ and $n$ edges, and let $g \in G(\qq)$.  Let $\underline{a}$ be the incidence matrix of $\mathcal{G}$, regarded as a sequence of length $n$ in $G'(\qq)$.
Then the number
of subgraphs of $\mathcal{G}$ of type $(\qq,g)$ is $N^{g}(G'(\underline{a}))$,
hence is $0$ or at least $2^{n+1-D(G'(\qq))}$.
\end{thm}


\subsection{Divisibility in Weighted Graphs}
\textbf{} \\ \\ \noindent
Let $G = G(\qq) = \bigoplus_{i=1}^r \Z/q_i \Z$ be an (arbitrary) finite commutative $p$-group. We give ourselves a sequence $\AA = \{A_i\}_{i=1}^{\infty}$ of finite nonempty subsets of $\Z$.  For $1 \leq j \leq r$, let $B_j \subset \Z/q_j\Z$ be nonempty subsets, and put $B = \prod_{i=1}^r B_j$, viewed as a subset of $G$.
We will give a graph theoretic application of the quantities
$D_{\AA}^{B}(G)$ and $N_{\AA}^{B}(G,n)$ which further generalizes the results of the previous section.
\\ \\
Let $\mathcal{G}$ be a finite graph with vertex set $V = \{1,\ldots,r\}$ and
edge set $E = \{1,\ldots,n\}$.  Put $\AA = \prod_{i=1}^n A_i$.  An element
$a \in \AA$ may be viewed as giving an integer weight $a_i$ to each edge $i$ of $\mathcal{G}$: we call this data an \textbf{A-weighted subgraph} of $\mathcal{G}$.
(The case $A_i = \{0,1\}$ for all $i$ recovers the usual notion of a subgraph.)
For a weighted subgraph $(\mathcal{G},a)$ and a vertex $j \in V$, the
\textbf{weighted degree} of $j$ is
\[ d_{\AA}(j) = \sum_{i \sim j} a_i, \]
that is, the sum of the weights of the edges incident to $j$.  A weighted subgraph $(\mathcal{G},a)$ is $B$-\textbf{divisible} if for all $1 \leq j \leq r$, we have $d_{\AA}(j) \in B_j \pmod{\Z/q_j\Z}$.
\\ \\
This setup is designed so that the number of $\AA$-weighted $B$-divisible subgraphs
is equal to the number of $\AA$-weighted $B$-sum subsequences of the sequence $\underline{a}$ in $G'(\qq)$ corresponding to the incidence matrix.
Thus we may apply the results of $\S 3.2$ to deduce the following result.

\begin{thm}
\label{VERYBIGAFK}
a) Let $G(\qq)$, $A = \prod_{i=1}^n A_i$, $B = \prod_{j=1}^r B_j$ be as above, and let $\mathcal{G}$ be a finite graph
with vertex set $V = \{1,\ldots,r\}$ and edge set $E = \{1,\ldots,n\}$.  Let $\underline{a}$ be the incidence matrix of $\mathcal{G}$, viewed as a sequence of length $n$ in $G'(\qq)$.  Then the number of $B$-divisible
$\AA$-weighted subgraphs of $\mathcal{G}$ is $N_{\AA}^B(\underline{a})$.  \\
b) If each $A_i$ contains $0$ and at least one element not divisible
by $q_r = \exp G(\qq)$ and each $B_j$ contains $0$, then there
is a nonempty $\AA$-weighted $B$-divisible subgraph of $\mathcal{G}$
whenever $n \geq D_{\AA}^B(G'(\qq))$.  \\
c) Let $p$ be a prime, let $1 \leq v_1 \leq \ldots \leq v_r \in \Z$ and put
$\qq = (p^{v_1},\ldots,p^{v_r})$.  Let $A_1,\ldots,A_n,B_1,\ldots,B_r$
each have the property that no two distinct elements are congruent modulo $p$.
Then: (i) the number of $A$-weighted $B$-divisible subgraphs of $\mathcal{G}$
is either $0$ or at least
\[ \mm\left(\# A_1,\ldots,\# A_n;\sum_{i=1}^n \# A_i - \sum_{j=1}^r (p^{v_j} - \# B_j)\right). \]
(ii) Suppose that $0$ lies in $A_i$ for all $1 \leq i \leq n$ and $0$ lies
in $B_j$ for all $1 \leq j \leq r$.  Then there is a nonempty $A$-weighted
$B$-divisible subgraph if
\[ \sum_{i=1}^n \left(\# A_i -1\right) > \sum_{j=1}^r (p^{v_j} - \# B_j). \]
\end{thm}

\begin{remark}
Theorem \ref{VERYBIGAFK}c)(ii) with $A_i = \{0,1\}$ recovers \cite[Thm. A.4]{Alon-Friedland-Kalai84}.
\end{remark}

\subsection{Polynomial Interpolation With Fat Targets}
\textbf{} \\ \\ \noindent
Our final application of the Main Theorem lies not in combinatorics but in algebra, specifically the problem of polynomial interpolation in commutative rings.

\begin{thm}
\label{LINEARMAINTHM}
Let $(\rr,\pp)$ be a finite, local principal ring with residue field $\F_q =
\rr/\pp$ and length $v$.  Let $f_1,\ldots,f_n \in \rr[t_1,\ldots,t_N]$ be an
$\rr$-linearly independent subset, and let $V = \langle f_1,\ldots,f_n \rangle$
be the $\rr$-module spanned by $f_1,\ldots,f_n$, so that every $f \in V$
may be uniquely written as
\[ f = \sum_{i=1}^n c_i(f) f_i, \ c_i(f) \in \rr. \]
Let $X = \{x_j\}_{j=1}^r \subset \rr^N$ be finite of cardinality $r$.  Let $A_1,\ldots,A_n,B_1,\ldots,B_r \subset \rr$ satisfy Condition (F).  For
$1 \leq j \leq r$, let $1 \leq v_j \leq v$.  \\
a) Let $\mathcal{S}$ be the set of $f \in V$ such that \\
(i) $c_i(f) \in A_i$ for all $1 \leq i \leq n$ and \\
(ii) $f(x_j) \in B_j \pmod{\pp^{v_j}}$ for all $1 \leq j \leq n$.  \\
Then $\# \mathcal{S} = 0$ or \[\# \mathcal{S} \geq \mm(\# A_1,\ldots,\# A_n;
\sum_{i=1}^n \# A_i - \sum_{j=1}^r (q^{v_j} - \# B_j)). \]
b) Suppose that $0$ is an element of each $A_i$ and $B_j$ and that
\[ \sum_{i=1}^n \# A_i - \sum_{j=1}^r (q^{v_j} - \# B_j) > n. \]
Then there is $0 \neq f \in \mathcal{S}$.
\end{thm}
\begin{proof}
a) Evaluation at $x \in X$ is a linear functional $L_i: \rr^X \ra \rr$.
Restricting each $L_i$ to $V$ gives a linear functional on $V$.
The basis $f_1,\ldots,f_n$ gives us an identification of $V$ with
$\rr^n$ under which $f = \sum_{i=1}^n c_i(f) f_i$ corresponds to
$(c_1(f),\ldots,c_n(f)) \in \rr^n$.  In this way we may view each $L_j$
as a linear polynomial on $\rr^n$.  For $f = (c_1(f),\ldots,c_n(f)) \in \rr^n$ the condition $L_j(f) \in
B_j \pmod{\pp^{v_j}}$ corresponds to $f(x_j) \in B_j \pmod{\pp^{v_j}}$.  So the Main Theorem applies. \\
b) The hypotheses imply that $0 \in \mathcal{S}$ and \\
$ \mm(\# A_1,\ldots,\# A_n;
\sum_{i=1}^n \# A_i - \sum_{j=1}^r (q^{v_j} - \# B_j)) \geq 2$.
\end{proof}

\begin{cor}
\label{TROIZANNIER}
For each $x \in \F_q^{\times}$, let $B_x$ be a subset of $\F_q$ containing $0$.
There is a nonzero polynomial $f \in \F_q[t_1]$ such that $f(0) = 0$,
$f(x) \in B_x$ for all $x \in \F_q^{\times}$ and
\[ \deg f \leq q - \frac{(\sum_{x \in \F_q^{\times}} \# B_x) - 1}{q-1}. \]
\end{cor}
\begin{proof}
Order the elements of $\F_q^{\times}$ as $x_1 = 0$,$x_2,\ldots,x_q$.  Apply Theorem \ref{LINEARMAINTHM}b)  with $\rr = \F_q$, $N = 1$, $f_1 = 1,f_2 = t_1,\ldots,f_{n+1} = t_1^n$, $X = \F_q$, $A_1 = \ldots = A_{n+1} = \F_q$, $B_1 = \{0\}$, $B_j = B_{x_j}$ for $2 \leq j \leq q$, $v_1 = \ldots = v_r = 1$: there is a nonzero polynomial of degree at most $n$ with $f(0) = 0$
and $f(x) \in S_x$ for all $x \in \F_q^{\times}$ if

\[ n+1 <  \sum_{i=1}^{n+1} \# A_i - \sum_{j=1}^q (q - \# B_j) = (n+1)q -
(q-1) - q(q-1) + \sum_{x \in \F_q^{\times}} \# B_x. \]
The latter inequality is equivalent to
\[ n \geq q - \frac{(\sum_{x \in \F_q^{\times}} \# B_x) + 1}{q-1}.  \qedhere \]
\end{proof}
\noindent
Corollary \ref{TROIZANNIER} is due to Troi and Zannier when $q = p$
is a prime \cite[Thm. 2]{Troi-Zannier97}.  Their proof is quite different: it uses Theorem \ref{TZB} and an auxiliary result using
integer-valued polynomials.  
Their argument seems not to carry over even to $\F_q$. 


\begin{thebibliography}{CCQWZ11}


\bibitem[AF93]{Alon-Furedi93} N. Alon and Z. F\"uredi, \emph{Covering the cube by affine hyperplanes}. Eur. J. Comb. 14 (1993), 79-–83.

\bibitem[AFK84]{Alon-Friedland-Kalai84} N. Alon, S. Friedland and G. Kalai, \emph{Regular subgraphs of almost regular graphs}. J. Combin. Theory Ser. B 37 (1984), 79–-91.


\bibitem[AGP94]{AGP94} W.R. Alford, A. Granville and C. Pomerance, \emph{There are infinitely many Carmichael numbers}. Ann. of Math. (2) 139 (1994), 703–-722.

\bibitem[AKLMRS]{AKLMRS91} N. Alon, D. Kleitman, R. Lipton, R. Meshulam,
M. Rabin and J. Spencer, \emph{Set systems with no union of cardinality 0 modulo m}. Graphs Combin. 7 (1991),  97–-99.

\bibitem[Al99]{Alon99} N. Alon, \emph{Combinatorial Nullstellensatz}.
Recent trends in combinatorics (M\'atrah\'aza, 1995).
Combin. Probab. Comput. 8 (1999), 7–-29.




\bibitem[Ax64]{Ax64} J. Ax, \emph{Zeroes of polynomials over finite fields}.
Amer. J. Math. 86 (1964), 255-–261.


\bibitem[BC15]{Brunyate-Clark15} A. Brunyate and P.L. Clark, 
    \emph{Extending the Zolotarev-Frobenius approach to quadratic reciprocity} Ramanujan J. 37 (2015), 25-–50.
    
\bibitem[Br11]{Brink11} D. Brink, \emph{Chevalley's theorem with restricted variables}. Combinatorica 31 (2011), 127-–130.


\bibitem[BS80]{Baker-Schmidt80} R.C. Baker and W.M. Schmidt, \emph{
Diophantine problems in variables restricted to the values $0$ and $1$}.
J. Number Theory 12 (1980), 460–-486.    
    


\bibitem[BS09]{Ball-Serra09} S. Ball and O. Serra, \emph{Punctured combinatorial Nullstellens\"atze}. Combinatorica 29 (2009), 511–-522.












\bibitem[CCQWZ11]{CCQWZ11}  G. J. Chang, S.-H. Chen., Y. Qu, G. Wang and
H. Zhang, \emph{On the number of subsequences with a given sum in a finite abelian group}. Electron. J. Combin. 18 (2011), no. 1, Paper 133, 10 pp.

\bibitem[CFS14]{Clark-Forrow-Schmitt14} P.L. Clark, A. Forrow and J.R. Schmitt,
\emph{Warning's Second Theorem With Restricted Variables}. To appear in Combinatorica.


\bibitem[Ch35]{Chevalley35} C. Chevalley, \emph{D\'emonstration d'une hypoth\`ese de M. Artin.} Abh. Math. Sem. Univ. Hamburg 11 (1935), 73–-75.

\bibitem[Cl14]{Clark14} P.L. Clark, \emph{The Combinatorial Nullstellens\"atze Revisited}.  Electronic Journal of Combinatorics. 	
Volume 21, Issue 4 (2014). Paper \#P4.15

\bibitem[DAGS12]{DAGS12} S. Das Adhikari, D.J. Grynkiewicz and Z.-W. Sun, \emph{
On weighted zero-sum sequences}. Adv. in Appl. Math. 48 (2012), 506–-527.

\bibitem[EBK69]{EBK69} P. van Emde Boas and D. Kruyswijk, \emph{A combinatorial problem on finite abelian groups, III}, Report ZW-
1969-008, Math. Centre, Amsterdam, 1969.

\bibitem[EGZ61]{EGZ61} P. Erd\H os, A. Ginzburg and A. Ziv, \emph{Theorem in the additive number theory}. Bull. Research Council Israel 10F (1961), 41--43.





\bibitem[HB11]{Heath-Brown11} D.R. Heath-Brown, \emph{On Chevalley-Warning theorems}. (Russian. Russian summary) Uspekhi Mat. Nauk 66 (2011), no. 2(398), 223--232; translation in Russian Math. Surveys 66 (2011), no. 2, 427-–436.

\bibitem[Hu68]{Hungerford68} T.W. Hungerford, \emph{On the structure of principal ideal rings}. Pacific J. of Math. 25 (1968), 543-–547.







\bibitem[MOS14]{MOS14} L.E. Marchan, O. Ordaz and 
W.A. Schmid, \emph{Remarks on the plus-minus weighted Davenport constant}. Int. J. Number Theory 10 (2014), 1219–-1239.

\bibitem[MN82]{Mead-Narkiewicz82} D.G. Mead and W. Narkiewicz, \emph{The apacity of $C_5$ and free sets in $C_m^2$}. Proc. Amer. Math. Soc. 84 (1982),  308-–310.

\bibitem[Ne71]{Necaev71} A.A. Ne\v caev, \emph{The structure of finite commutative rings with unity}. Mat. Zametki 10
(1971), 679--688.


\bibitem[O69a]{Olson69a} J.E. Olson, \emph{A combinatorial problem on finite Abelian groups}. I. J. Number Theory 1 (1969), 8–-10.

\bibitem[Ol69b]{Olson69b} J.E. Olson, \emph{A combinatorial problem on finite Abelian groups. II.} J. Number Theory 1 (1969), 195-–199.




\bibitem[Sc74]{Schanuel74} S.H. Schanuel, \emph{An extension of Chevalley's theorem to congruences modulo prime powers}. J. Number Theory 6 (1974), 284-–290.


\bibitem[Sc08]{Schauz08a} U. Schauz, \emph{Algebraically solvable problems: describing polynomials as equivalent to explicit solutions}.
Electron. J. Combin. 15 (2008), no. 1, Research Paper 10, 35 pp.



\bibitem[Sh56]{Shannon56} C.E. Shannon, \emph{The zero error capacity of a noisy channel.}
Institute of Radio Engineers, Transactions on Information Theory, IT-2, 1956, pp. 8–-19.


\bibitem[St77]{Stein77} S.K. Stein, \emph{Modified linear dependence and the capacity of a cyclic graph}. Linear Algebra and Appl. 17 (1977), 191–-195.

\bibitem[Th07]{Thangadurai07} R. Thangadurai, \emph{A variant of Davenport's constant}. Proc. Indian Acad. Sci. Math. Sci. 117 (2007), 147–-158.

\bibitem[TZ97]{Troi-Zannier97} G. Troi and U. Zannier, \emph{On a theorem of J. E. Olson and an application (vanishing sums in finite abelian p-groups)}. Finite Fields Appl. 3 (1997), 378–-384.


\bibitem[Wa35]{Warning35} E. Warning, \emph{Bemerkung zur vorstehenden Arbeit von Herrn Chevalley}.  Abh. Math. Sem. Hamburg 11 (1935), 76–-83.


\bibitem[Wi06]{Wilson06} R.M. Wilson, \emph{Some applications of polynomials in combinatorics}. IPM Lectures, May, 2006.

\end{thebibliography}
\end{document}